\documentclass[a4paper]{amsart}

\usepackage{amssymb,pstricks,amscd,epsfig}
\usepackage[T1]{fontenc}
\usepackage[totalwidth=16cm,totalheight=22cm]{geometry}
\usepackage{graphicx}

\begin{document}
\newtheorem{cor}{Corollary}[section]
\newtheorem{theorem}[cor]{Theorem}
\newtheorem{prop}[cor]{Proposition}
\newtheorem{lemma}[cor]{Lemma}
\newtheorem{sublemma}[cor]{Sublemma}
\theoremstyle{definition}
\newtheorem{defi}[cor]{Definition}
\theoremstyle{remark}
\newtheorem{remark}[cor]{Remark}
\newtheorem{example}[cor]{Example}

\newcommand{\cA}{{\mathcal A}}
\newcommand{\cD}{{\mathcal D}}
\newcommand{\cE}{{\mathcal E}}
\newcommand{\cF}{{\mathcal F}}
\newcommand{\cG}{{\mathcal G}}
\newcommand{\cM}{{\mathcal M}}
\newcommand{\cN}{{\mathcal N}}
\newcommand{\cT}{{\mathcal T}}
\newcommand{\cW}{{\mathcal W}}
\newcommand{\cML}{{\mathcal M\mathcal L}}
\newcommand{\cFML}{{\mathcal F\mathcal M\mathcal L}}
\newcommand{\cGH}{{\mathcal G\mathcal H}}
\newcommand{\cQF}{{\mathcal Q\mathcal F}}
\newcommand{\dwp}{d_{WP}}
\newcommand{\C}{{\mathbb C}}
\newcommand{\N}{{\mathbb N}}
\newcommand{\R}{{\mathbb R}}
\newcommand{\Z}{{\mathbb Z}}
\newcommand{\Kt}{\tilde{K}}
\newcommand{\Mt}{\tilde{M}}
\newcommand{\dr}{{\partial}}
\newcommand{\betab}{\overline{\beta}}
\newcommand{\kappab}{\overline{\kappa}}
\newcommand{\pib}{\overline{\pi}}
\newcommand{\taub}{\overline{\tau}}
\newcommand{\gb}{\overline{g}}
\newcommand{\hb}{\overline{h}}
\newcommand{\ub}{\overline{u}}
\newcommand{\Kb}{\overline{K}}
\newcommand{\Sigmab}{\overline{\Sigma}}
\newcommand{\gd}{\dot{g}}
\newcommand{\Id}{\dot{I}}
\newcommand{\diff}{\mbox{Diff}}
\newcommand{\dev}{\mbox{dev}}
\newcommand{\devb}{\overline{\mbox{dev}}}
\newcommand{\devt}{\tilde{\mbox{dev}}}
\newcommand{\vol}{\mbox{Vol}}
\newcommand{\hess}{\mbox{Hess}}
\newcommand{\cb}{\overline{c}}
\newcommand{\db}{\overline{\partial}}
\newcommand{\hgr}{h_{gr}}
\newcommand{\Sigmat}{\tilde{\Sigma}}

\newcommand{\cunc}{{\mathcal C}^\infty_c}
\newcommand{\cun}{{\mathcal C}^\infty}
\newcommand{\dd}{d_D}
\newcommand{\dmin}{d_{\mathrm{min}}}
\newcommand{\dmax}{d_{\mathrm{max}}}
\newcommand{\Dom}{\mathrm{Dom}}
\newcommand{\dn}{d_\nabla}
\newcommand{\ded}{\delta_D}
\newcommand{\delmin}{\delta_{\mathrm{min}}}
\newcommand{\delmax}{\delta_{\mathrm{max}}}
\newcommand{\hmin}{H_{\mathrm{min}}}
\newcommand{\maxi}{\mathrm{max}}
\newcommand{\oL}{\overline{L}}
\newcommand{\oP}{{\overline{P}}}
\newcommand{\xb}{{\overline{x}}}
\newcommand{\yb}{{\overline{y}}}
\newcommand{\Ran}{\mathrm{Ran}}
\newcommand{\tgamma}{\tilde{\gamma}}
\newcommand{\cotan}{\mbox{cotan}}
\newcommand{\area}{\mbox{Area}}
\newcommand{\lambdat}{\tilde\lambda}
\newcommand{\xt}{\tilde x}
\newcommand{\Ct}{\tilde C}
\newcommand{\St}{\tilde S}

\newcommand{\sh}{\mathrm{sinh}\,}
\newcommand{\ch}{\mathrm{cosh}\,}
\newcommand{\tr}{\mathrm{tr}\,}

\newcommand{\II}{I\hspace{-0.1cm}I}
\newcommand{\III}{I\hspace{-0.1cm}I\hspace{-0.1cm}I}
\newcommand{\note}[1]{{\small {\color[rgb]{1,0,0} #1}}}

\title{The renormalized volume and the volume of the convex core of quasifuchsian manifolds}

\author{Jean-Marc Schlenker}
\address{University of Luxembourg, Campus Kirchberg,
Mathematics Research Unit, BLG,
6, rue Richard Coudenhove-Kalergi, 
L-1359 Luxembourg}
\email{jean-marc.schlenker@uni.lu}

\date{Revised version --- November 2014}

\begin{abstract}
We show that the renormalized volume of a quasifuchsian hyperbolic 3-manifold is equal, up
to an additive constant, to the volume of its convex core. 
We also provide a precise upper bound on the renormalized volume in terms of the
Weil-Petersson distance between the conformal structures at infinity.
As a consequence we show that holomorphic disks in Teichm\"uller space 
which are large enough must have ``enough'' negative curvature.
\end{abstract}

\maketitle

\section{Results}

\subsection{Notations}

In all the paper we consider a closed surface $S$ of genus 
$g\geq 2$, and we call $\cT_S$ the Teichm\"uller space of $S$.
Given a complex structure $c\in \cT_S$, we denote by $Q_c$ the vector space of
holomorphic quadratic differentials on $(S,c)$.
We also call $\cG_S$ the space of quasifuchsian metrics on 
$S\times \R$, considered up to isotopy.

The Bers Simultaneous Uniformization
Theorem provides a homeomorphism between $\cT_{S\cup \overline{S}}$
and $\cG_S$. So $\cG_S$
is parameterized by $\cT_+\times \cT_-$, where
$\cT_+$ and $\cT_-$ are two copies of $\cT_S$ corresponding
respectively to the upper and lower boundaries at infinity
of $S\times \R$.

For $q\in \cG$, we denote by $V_R(q)$ the renormalized volume of $(S\times \R,g)$ (as defined
in Section 3 following e.g. \cite{review}), while 
$C(q)$ is the convex core of $(S\times \R,g)$ and $V_C(q)$ is its volume.

\subsection{Comparing the renormalized volume to the 
volume of the convex core}

The first result presented here is a precise comparison between the volume of the
convex core and the renormalized volume of a quasifuchsian hyperbolic manifold.

\begin{theorem} \label{tm:compare}
There exists a constant $C_g>0$, depending only on the genus $g$ of $S$, as follows. 
Let $q\in \cG_{S}$ be a quasifuchsian metric. Then
$$ V_R(q) \leq V_C(q) - (1/4)L_m(l) + \frac{\pi \log(2)}2|\chi(\partial M)|
\leq V_R(q)+C_g~, $$
where $l$ is the measured bending lamination of the boundary of $C(q)$,
$m$ is its induced metric, and $L_m(l)$ is the length of $l$ with respect
to $m$.
Equality in the first inequality occurs exactly when $q$ is Fuchsian.
\end{theorem}

Note that the quantity $V_C(q)-(1/4)L_m(l)$ appearing in this theorem 
can be interpreted as the half-sum of the volume and the ``dual volume''
of the convex core (as appearing e.g. in \cite{cp}). Recall also that 
there is a bound on $L_m(l)$ depending only on the genus $g$, see \cite{bridgeman}.
The constant $\frac{\pi \log(2)}2|\chi(\partial M)|$ appearing in the middle term
is related to normalization issues in the definition of the renormalized volume,
and could be suppressed by a suitable normalization.

Theorem \ref{tm:compare}
can be extended to convex co-compact manifolds with incompressible boundary.
We do not elaborate on this here but the statement and proof of the extension
should appear clearly from the proof of Theorem \ref{tm:compare} below.

\subsection{An upper bound on the renormalized volume}

The next result is a precise upper bound on the renormalized volume of
a quasifuchsian manifold, in terms of the Weil-Petersson distance between
its conformal structures at infinity.

\begin{theorem} \label{tm:upper}
For any quasifuchsian metric $q$ on $S\times \R$,
\begin{equation}
V_R(q) \leq 3\sqrt{\pi (g-1)} \dwp(c_-,c_+)~, 
\end{equation}
where $c_-$ and $c_+$ are the conformal structure at infinity of $q$ and 
$d_{WP}$ is the Weil-Petersson distance.
\end{theorem}

This result is related, through Theorem \ref{tm:compare}, to the following result of 
Brock on the comparison between the Weil-Petersson distance $\dwp(c_-,c_+)$ and 
the volume of the convex core $V_C(q)$.

\begin{theorem}[Brock \cite{brock:2003}] \label{tm:brock}
For a given surface $S$, $V_C(c_-,c_+)$ is comparable to $\dwp(c_-,c_+)$,
that is, there are constants $k_1,k_2>0$ such that 
\begin{equation} \label{eq:brock}
\frac{\dwp(c_-,c_+)}{k_1}-k_2 \leq V_C(c_-,c_+)\leq k_1 \dwp(c_-,c_+)+k_2~. 
\end{equation}
\end{theorem}

We will recover a precise form of the upper bound on $V_C$ from Theorem \ref{tm:compare}
and Theorem \ref{tm:upper}. 
Recall that Bridgeman \cite[Proposition 2]{bridgeman} proved that the length
of the measured bending lamination, $L_m(l)$, is bounded from above by a 
constant $K_g$ depending only on the genus of $S$. The following statement then follows 
immediately from Theorem \ref{tm:upper} and Theorem \ref{tm:compare}.

\begin{cor} \label{cor:brock}
There exists a constant $K_g$ depending only on the genus of $S$ such that
$$ V_C(c_-,c_+) \leq 3\sqrt{\pi(g-1)} \dwp(c_-,c_+)+K_g~. $$
\end{cor}

This is a more precise version of the second inequality in (\ref{eq:brock}).

\subsection{Holomorphic disks in Teichm\"uller space}

An interesting difference between Theorem \ref{tm:upper} and Theorem \ref{tm:brock} is that the
renormalized volume has, in addition to the ``coarse'' properties of the volume of the convex
core, some remarkable analytic properties related to the Weil-Petersson 
metric on Teichm\"uller space. We will use this and the upper bound on $V_C$
to obtain a statement on the global geometry of the Weil-Petersson metric on Teichm\"uller
space, more precisely on the maximal radius of holomorphic disks with curvature bounded from
below (see Theorem \ref{tm:disk}). The statement below
is probably not optimal, it is more of an indication of the
type of results one can obtain using the remarkable properties of the renormalized
volume. 

\begin{theorem} \label{tm:disk}
There exists a smooth, increasing function $\phi:[0,1)\rightarrow \R_{\geq 0}$
with $\phi(0)=0$, $\phi'(0)=2$ and $\lim_1\phi=\infty$ as follows. 
Let $k>0$ be such that $3k^2\sqrt{\pi(g-1)}<2$.
There is no  immersed holomorphic disk $D$ of radius $\phi(3k^2\sqrt{\pi(g-1)}/2)/k$ 
in $\cT_S$ with curvature $K\geq -k^2$. 
\end{theorem}

The radius here is the radius of the metric induced on $D$ by the Weil-Petersson
metric on $\cT_S$. The expression of the function $\phi$ can be obtained by solving a simple
differential equation, see Section 6. This kind of statement is presumably most interesting
close to the boundary of $\cT_S$ for the Weil-Petersson metric, where the sectional curvature tends
to be close to $0$ in many directions, see \cite{wolpert-review,wolpert:geodesic-length}.
It should be compared to (and is related to) the isoperimetric inequality for complex submanifolds
of Teichm\"uller space discovered by McMullen, see \cite[Corollary 1.3]{McMullen}.

\subsection{From the boundary of the convex core to infinity}

There is a clear parallel between data ``at infinity'' of a quasifuchsian (or more generally
a convex co-compact hyperbolic manifold) and corresponding data on
the boundary of the convex core. The results presented here contribute to clarify
and extend those analogies. A synthetic presentation of this correspondence is
shown in Table \ref{tab:dictionary}. Clearly, the conformal metric at infinity corresponds
to the induced metric on the boundary of the convex core. The analog at infinity of 
the measured bending lamination of the boundary of the convex core is less obvious, but
it appears to be the ``second fundamental form at infinity'' $\II^*$ as defined in Section 
\ref{ssc:variational}, or more precisely its traceless part $\II^*_0$. This second fundamental 
form at infinity gives the second term in the asymptotic expansion of the metric at infinity
associated to an equidistant foliation, see below, and in this sense, too, it is analogous
to the measured bending lamination on the boundary of the convex core. Other
elements in this analogy between the data on the boundary of the convex core and the
data at infinity are resumed in Table \ref{tab:dictionary}, the different lines are 
explained below.

\begin{table}[ht]
  \centering
  \begin{tabular}{|c|c|}
    \hline
    Boundary of the convex core & At infinity \\
    \hline\hline
    Induced metric $m$ & Conformal/hyperbolic metric at infinity \\
    \hline 
    Measured bending lamination $l$ & $\II^*_0$ \\
    \hline
    Bound on $L_m(l)$ \cite{bridgeman} & Theorem \ref{tm:nehari} \\
    \hline 
    Volume of the convex core & Renormalized volume \\
    \hline
    Bonahon's Schl\"afli formula & Proposition \ref{pr:variation} \\
    \hline
    Brock's upper bound on $V_C$ \cite{brock:2003} & Theorem \ref{tm:upper} \\
    \hline
  \end{tabular}
  \caption{Infinity vs the boundary of the convex core}
  \label{tab:dictionary}
\end{table}


\subsection*{Acknowledgements}

The author would like to thank Juan Souto for useful conversations related to
the results presented here, and Colin Guillarmou and Fr\'ed\'eric Paulin for providing relevant 
references necessary at some points of the arguments. Thanks also to Jiming
Ma, Greg McShane, Sergiu Moroianu and to an anonymous referee for many helpful comments leading to improvement 
over the first version of this text.

This revised version corrects errors in some coefficients in the published version. The author is 
particularly grateful to Curt McMullen for noting those errors, and for many illuminating comments
and explanations.

\section{Conformal changes of metrics}

We gather in this short section some basic and well-known results on conformal changes of
metrics on surfaces. 

\begin{lemma}
Let $h$ be a Riemannian metric on $S$, of curvature $K$. Let 
$u:S\rightarrow \R$, let $\hb=e^{2u}h$, and let $\Kb$ be the
curvature of $\hb$. Then 
$$ \Kb = e^{-2u}(K+\Delta u)~. $$
\end{lemma}

Here $\Delta$ is the ``geometer's'' Laplacian, that is, it is non-positive
at the minima. The proof of this lemma can be found e.g. in \cite[Section 1]{Be}

\begin{lemma} \label{lm:conf-compare}
Let $h$ be a hyperbolic metric on $S$, and let $u:S\rightarrow \R$ be such
that $e^{2u}h$ has curvature $\Kb\geq -1$. Then $u\geq 0$.    
\end{lemma}

\begin{proof}
Let $x\in S$ be a point where $u$ is minimal. Then $\Delta u\leq 0$ at $x$, 
and $\Kb=e^{-2u}(-1+\Delta u)\leq -e^{-2u}$. Since $\Kb\geq -1$ it follows that
$u\geq 0$ at $x$. 
\end{proof}

\section{The renormalized volume}

\subsection{Some background}

Consider a Poincar\'e-Einstein manifold $M$, that is, a manifold $M$ with boundary,
with an Einstein metric $g$ on the interior of $M$ which can be written as 
$$ g = \frac{\gb}{\rho^2}~, $$
where $\rho$ is a function which vanishes on $\partial M$ with $\| d\rho\|_{\gb} =1$
on $\partial M$. 

The volume of $(M,g)$ is infinite, however there is a well-defined
way to define a ``regularized'' version of this volume, called the renormalized
volume of $M$, which is finite (see e.g. \cite{graham,graham-witten}). If the
dimension of $M$ is odd, it depends on the choice of a metric in the conformal 
class of the boundary of $(M,g)$, while if the dimension of $M$ is even it
is canonically defined.

If $M$ is a convex co-compact hyperbolic manifold, it is Poincar\'e-Einstein 
according to the definition above, so that the definition of its renormalized
volume applies. The fact that the metric has constant curvature makes it possible
to give an explicit description of the geometry of the leaves of an equidistant
foliation, see \cite{epstein-duke}.
Since the dimension is odd, this renormalized volume depends on
the choice of a metric at infinity, however there is a canonical choice available
for this metric: the unique metric of constant curvature $-1$ in the conformal
class at infinity.

This renormalized volume turns out to be strongly related to the Liouville
functional previously studied by Takhtajan, Zograf and Teo  
\cite{TZ-schottky,takhtajan-teo}, see \cite{Holography}. In particular it has
some remarkable relations to the Weil-Petersson metric on Teichm\"uller space. 
Moreover, there is a simpler definition specific to dimension $3$. We recall below
this definition and the key properties of this $3$-dimensional renormalized volume 
in a form suitable for the applications considered here.

We can mention that the renormalized volume is the real part of a complex quantity
with an imaginary part related to the Chern-Simons invariant, see \cite{guillarmou-moroianu}.
Some of the properties of the 3-dimensional renormalized volume used here actually
extend in some measure to higher (odd) dimensions, see \cite{renormvol}.

\subsection{Metrics at infinity and equidistant foliations}

We now consider a quasifuchsian hyperbolic $3$-manifold $(M,g)$. 

\begin{defi}
Let $E$ be an end of $M$. An {\it equidistant foliation} in $E$ is a
foliation of a neighborhood of infinity in $E$ by convex surfaces, $(S_r)_{r\geq r_0}$, 
for some $r_0>0$, such that, for all $r'>r\geq r_0$, $S_{r'}$ is between
$S_r$ and infinity, and at constant distance $r'-r$ from $S_r$.
\end{defi}

Two equidistant foliations in $E$ will be identified if they coincide in
a neighborhood of infinity. In this case they can differ only by the 
first value $r_0$ at which they are defined.

Note that given an equidistant foliation $(S_r)_{r\geq r_0}$ and given 
$r'>r\geq 0$, there is a natural identification between $S_r$ and $S_{r'}$,
obtained by following the normal direction from $S_r$ to $S_{r'}$. This
identification will be implicitly used below. 

\begin{defi} \label{df:I*}
Let $M$ be a convex co-compact hyperbolic manifold, let $E$ be an end of 
$M$, and let $(S_r)_{r\geq r_0}$ be an equidistant foliation in $E$. The
{\it metric at infinity} associated to $(S_r)_{r\geq r_0}$ is the metric:
$$ I^* = \lim_{r\rightarrow \infty} 2e^{-2r} I_r~, $$
where $I_r$ is the induced metric on $S_r$. 
\end{defi}

We will make use of the following proposition. The first part is quite
elementary (see e.g. \cite{volume}) while the second part can be found,
in the more general setting of conformally compact Einstein manifolds, 
in \cite[Lemma 2.1]{graham} (see also \cite[Lemma 5.2]{graham-lee} and 
\cite[Lemme 2.1.2]{djadli-guillarmou-herzlich}).

\begin{prop} \label{pr:33}
$I^*$ always exists, and it is in the conformal class at infinity of $E$. 

Let $M$ be a convex co-compact hyperbolic manifold, let $E$ be an end of 
$M$, and let $h$ be a Riemannian metric in the conformal class at infinity
of $E$. There is a unique equidistant foliation in $E$ such that the
associated metric at infinity is $h$.
\end{prop}

This equidistant foliation can be defined from $I^*$ in terms of envelope of a family
of horospheres, see \cite{Eps}, we recall this construction here. Consider the
hyperbolic space $H^3$ as the universal cover of $M$, then $I^*$ lifts to a 
metric on the domain of discontinuity $\Omega$ of $M$, in the canonical conformal class of
$\partial_\infty H^3$. Let $x\in \Omega$. For each $y\in H^3$, the visual
metric $h_y$ on $\partial_\infty H^3$ is conformal to $I^*$. Let $H_{x,r}$ be
the set of points $y\in H^3$ such that $h_y\geq e^{2r}I^*$ at $x$ --- it is not difficult
to check that $H_{x,r}$ is a horoball intersecting $\partial_\infty H^3$ at $x$,
and the lift of $S_r$ to $H^3$ happens to be equal to the boundary of the 
union of the $H_{x,r}$, for $x\in \Omega$.

\subsection{Definition and first variation of $W$}
\label{ssc:first}

To define the renormalized volume of a quasifuchsian manifold, below, we first introduce
a modified volume of convex subsets. We consider a quasifuchsian manifold $M$ and 
a convex subset $N$ of $M$ with smooth boundary --- here ``convex'' means that, whenever
$\gamma\subset M$ is a geodesic segment with endpoints in $N$, $\gamma\subset N$.
We will define first (in Definition \ref{df:W}) a modified volume of $N$, and then
use this modified volume, for a particular choice of a convex subset of $M$,
to define the renormalized volume of $M$ (Definition \ref{df:renormvol}).

\begin{defi} \label{df:W}
Let $N\subset M$ be a convex subset. We define
$$ W(N) = V(N) -\frac{1}{4}\int_{\dr N} H da $$
where $H$ is the mean curvature of $\partial N$ and $da$ is the area form of
its induced metric.
\end{defi}

The first variation of this modified volume is given in \cite{volume}, based on an earlier
variation formula for deformations of Einstein manifolds with boundary \cite{sem,sem-era}.
Here we consider a first-order deformation of the hyperbolic metric on $N$, and
denote by $I'$ and $\II'$, respectively, the corresponding first-order variations of the
induced metric and second fundamental form on the boundary of $N$.

\begin{lemma} \label{lm:var1}
Under a first-order deformation of $N$,
\begin{equation} \label{eq:var1}
   W' = \frac{1}{4}\int_{\dr N} \langle \II' - \frac{H}{2} I',I\rangle da~. 
\end{equation}
\end{lemma}

The scalar product appearing in \eqref{eq:var1} between symmetric bilinear
forms is the usual extension to tensors of the Riemannian scalar product on 
$T\partial N$ associated to the induced metric $I$.

The following lemma is a direct consequence of Lemma \ref{lm:var1}, see below.

\begin{lemma} \label{lm:linear}
Let $r\geq 0$, let $N_r$ be the set of points of $M$ at distance at most $r$
from $N$. Then $W(N_r) = W(N) - \pi r\chi(\dr M)$. 
\end{lemma}

\begin{proof}
For $s\in [0,r]$, let $N_s$ be the set of points of $M$ at distance at most
$s$ from $N$, and let $w(s)=W(N_s)$. Let $I_s, \II_s, \III_s$ and $B_s$ be the induced 
metric, second and third fundamental forms and the Weingarten operator of $\dr N_s$.  
According to standard differential geometry formulas,
$$ I_s' =2\II_s~, ~~ \II_s' = \III_s + I_s~. $$ 
Lemma \ref{lm:var1} therefore shows that 
$$ W(N_s)' = \frac{1}4 \int_{\dr N} \langle \III_s+I_s-H_s \II_s, I_s\rangle da_s = 
\frac{1}4 \int_{\dr N} \tr(B_s^2)+2-H^2_s da_s = $$
$$ = \frac{1}4 \int_{\dr N} 2-2\det(B_s) da_s = 
\frac{1}2\int_{\dr N} -K da_s = -\pi\chi(\dr N)~. $$
\end{proof}

Consider a Riemannian metric $h$ on $\partial M$ in the conformal class at 
infinity of $M$. By Proposition \ref{pr:33} there is a unique equidistant 
foliation $(s_r)_{r\geq r_0}$ of $M$ near infinity such that the associated metric is $h$. 
For $r$ large enough, the surfaces $S_r$ bound a convex subset of $M$, so that
Definition \ref{df:W} can be applied.

\begin{defi} \label{df:WW}
Let $h$ be a metric on $\dr M$, in the conformal class at infinity. 
Let $(S_r)_{r\geq r_0}$ be the equidistant foliation at infinity associated
to $h$. We define $W(M,h) := W(S_r)+ \pi r\chi(\dr M)$, for any choice 
of $r\geq r_0$. 
\end{defi}

This definition does not depend on the choice of $r$ by Lemma \ref{lm:linear}.

\begin{cor} \label{cr:add}
For any $\rho\in \R$, $W(M,e^{2\rho}h)=W(M,h) - \pi\rho \chi(\dr M)$.  
\end{cor}

\subsection{Variational formula for $W$ from infinity} 
\label{ssc:variational}

Given an equidistant foliation of the end $E$, the hyperbolic metric $q$ actually
takes a remarkably simple form, see \cite{volume,review}. It can be written as
$$ q = dr^2 + \frac 12(e^{2r} I^* + 2\II^* + e^{-2r}\III^*)~, $$
where $I^*$ is the metric at infinity called $h$ above, and $\II^*$ and $\III^*$
are analogs at infinity of the second and third fundamental forms of a surface.
More precisely, there is a unique bundle morphism $B^*:TS\rightarrow TS$ which
is self-adjoint for $I^*$ and such that 
$$ \II^* = I^*(B^*\cdot, \cdot)~,~~ \III^* = I^*(B^*\cdot,B^*\cdot)~. $$
Then $B^*$ satisfies the Codazzi equation $d^{\nabla^*}B^*=0$, where $\nabla^*$
is the Levi-Civita connection of $I^*$, and an analog of the Gauss equation, 
$\tr(B^*)=-K^*$, where $K^*$ is the curvature of $I^*$.

Consider now, as in Section \ref{ssc:first}, a convex subset $N\subset M$ with
smooth boundary, and the equidistant foliation of $M\setminus N$ by surfaces at constant
distance from $N$. 
The data at infinity $I^*, \II^*,\III^*$ can be written in terms of the data $I, \II, \III$
on the boundary of $N$ as follows (see \cite[Section 5]{volume}): if $E$ is the identity
on $T\partial N$ and $B$ is the shape operator of $TN$, then
$$ I^* = \frac 12 I((E+B)\cdot, (E+B)\cdot)~, ~~
\II^* =  \frac 12 I((E+B)\cdot, (E-B)\cdot)~, $$
$$ \III^* =  \frac 12 I((E-B)\cdot, (E-B)\cdot)~. $$
Conversely, a direct computation (see also \cite[Section 5]{volume}) shows that 
the same formulas express the data on the boundary of $N$ in terms of the data at infinity:
$$ I = \frac 12 I^*((E+B^*)\cdot, (E+B^*)\cdot)~, ~~ 
\II =  \frac 12 I^*((E+B^*)\cdot, (E-B^*)\cdot)~, $$
$$ \III =  \frac 12 I^*((E-B^*)\cdot, (E-B^*)\cdot)~. $$ 
Using those transformation formulas, one
can write the first-order variation of $W$ in terms of the data at infinity,
and it turns out to be remarkably similar to the variation formula (\ref{eq:var1})
in terms of the data on $\partial N$, see \cite[Section 6]{volume}.
\begin{equation} \label{eq:var2}
   W' = -\frac{1}{4}\int_{\dr N} \langle {\II^*}' - \frac{H^*}{2} {I^*}',I^*\rangle da^*~. 
\end{equation}
Here ${I^*}'$ and ${\II^*}'$ are the first-order variations of $I^*$ and 
$\II^*$, while $H^*=\tr(B^*)$ and $da^*$ is the area form of $I^*$. 

\subsection{The renormalized volume}

We can now give the definition of the renormalized volume of $M$.

\begin{defi} \label{df:renormvol}
The renormalized volume $V_R$ of $M$ is defined as equal to $W(h)$ when 
the metric at infinity $h$ is the unique metric of constant curvature $-1$
in the conformal class at infinity of $M$.
\end{defi}

Another possible definition is as the maximum of $W(M,h)$ over all metrics
$h$ in the conformal class at infinity of $M$, under the condition that the
area of $h$ is equal to $-2\pi \chi(\partial M)$, see \cite{volume}.

\subsection{A variational formula for the renormalized volume}

Consider now a first-order deformation of $M$, specified --- through the
Bers Double Uniformization Theorem --- by a first-order deformation of
the conformal structure at infinity, considered as a point in the 
Teichm\"uller space of $\partial M$. 

\begin{prop} \label{pr:variation}
Under a first-order deformation of the hyperbolic structure on $M$,
\begin{equation}
  \label{eq:schlafli-VR}
  dV_R = - \frac{1}{4} \int_{\partial M}\langle \II^*_0, \Id^*\rangle_{I^*} da_{I^*}~.
\end{equation}
\end{prop}

Here $\langle, \rangle_{I^*}$ is the extension to symmetric 2-tensors of the Riemannian
metric $I^*$ on $T\partial M$. 
Proposition \ref{pr:variation}
follows by a simple computation from Equation \eqref{eq:var2}, see \cite[Lemma 8.5]{volume}.

It should be pointed out that Proposition \ref{pr:variation} has a rather simple translation in terms of complex analysis. Since $\II^*_0$ is Codazzi and traceless, it is the real part of a holomorphic quadratic differential, which is minus the Schwarzian derivative $q$ of the uniformization map, see Section \ref{ssc:bers} below. Moreover, any first-order deformation $\Id^*$ of the hyperbolic metric at infinity determines a first-order variation of the underlying complex structure, and therefore a Beltrami differential $\mu$.

\begin{cor}
Equation \eqref{eq:schlafli-VR} can then be written as:
\begin{equation}
  \label{eq:schlafli-simple}
dV_R = Re\left(\langle q,\mu\rangle\right) = \int_{\partial M} Re(q\mu)~,
\end{equation}
where $\langle, \rangle$ is the natural pairing between holomorphic quadratic differentials and Beltrami differentials.
\end{cor}

\begin{proof}
The computation needed to go from \eqref{eq:schlafli-VR} to \eqref{eq:schlafli-simple} is local. We choose a complex coordinate $z=x+iy$ adapted to $I^*$, that is, such that $I^*=dx^2+dy^2$ at $z=0$. Let $\mu=(\mu+i\mu_1)\frac{\bar{dz}}{dz}$, and $q=(q_0+iq_1)dz^2$. Then 
$$ \II^*_0=-Re(q)=-(q_0(dx^2-dy^2)-2q_1dx dy)~, $$
while the traceless part of the first-order variation of $I^*$ is equal to
\begin{eqnarray*}
\Id^*_0 & = & \frac{d}{dt}_{|t=0}|dz(1+\mu)|^2 \\
& = & \frac{d}{dt}_{|t=0}|dz+(\mu_0+i\mu_1)\bar{dz}|^2 \\
& = & 2Re((\mu_0+i\mu_1)\overline{dz}) \\
& = & 2(\mu_0(dx^2-dy^2)+2\mu_1 dx dy)~.  
\end{eqnarray*}
As a consequence, 
\begin{eqnarray*}
\langle \II^*_0,\Id^*\rangle_{I^*} & = & \langle -(q_0(dx^2-dy^2)-2q_1dx dy), 2(\mu_0(dx^2-dy^2)+2\mu_1 dx dy\rangle_{I^*} \\
& = & -4 (\mu_0 q_0 -\mu_1q_1)~,  
\end{eqnarray*}
so that 
$$ \langle \II^*_0,\Id^*\rangle_{I^*} da_{I^*}= -4Re(q\mu)~. $$
The result follows by integrating this equality.
\end{proof}

\subsection{Comparing metrics at infinity}

\begin{prop} \label{pr:conf-change}
If $h,h'$ are two metrics of non-positive curvature in the conformal class at infinity on $\dr M$
and $h'$ is everywhere at least as large as $h$, then $W(M,h')\geq W(M,h)$, with
equality if and only if $h=h'$.  
\end{prop}

The proof of this proposition will follow the next two lemmas.

\begin{lemma} \label{lm:hh'}
Let $h,h'$ be two metrics in the conformal class at infinity on $\dr M$.
Suppose that $h'$ is everywhere at least as large as $h$. Let $r$ be 
large enough so that both $S_{h,r}$ and $S_{h',r}$ are well-defined. Then
$S_{h,r}$ is in the interior of $S_{h',r}$.
\end{lemma}

\begin{proof}
We have seen above (just after Proposition \ref{pr:33}) that $S_{h,r}$ can be defined as 
the boundary of the complement of the union of the horoballs
associated to $h$ of ``radius'' $r$ at points of $\dr M$.  
Since $h'$ is everywhere at least as large as $h$, the horoball
associated to $h'$ of radius $r$ is at each point contained in the
horoball associated to $h$ of radius $r$. It follows that 
$M_{h,r}\subset M_{h',r}$.
\end{proof}

\begin{defi}
Let $E$ be a hyperbolic end, let $S,S'$ be two surfaces in $E$
such that $S$ is contained in the ``interior'' of $S'$. We set
$$ W(S,S') = V(S,S') - \frac{1}{4} \int_{S'} Hda +
\frac{1}{4}\int_S Hda~. $$
\end{defi}

It follows from this definition that if $E$ is an end of $M$
containing two surfaces $S$ and $S'$ with $S$ contained in the
interior of $S'$, if $h$ is the metric at infinity in the conformal
class at infinity on $\dr M$
corresponding to $S$ in $\dr E$ and $h'$ is another metric 
in the conformal class at infinity on $\dr M$, equal to $h$
except that it corresponds to $S'$ in $\dr E$, then
$$ W(M,h')=W(M,h)+W(S,S')~. $$

\begin{lemma}
If $S$ is contained in the interior of $S'$ and the induced metrics on both $S$
and $S'$ have non-positive curvature, then $W(S,S')\geq 0$,
with equality only if $S=S'$.
\end{lemma}

\begin{proof}
We first construct a smooth one-parameter family of surfaces, $(S_t)_{t\in [0,1]}$, with 
$S_0=S$, $S_1=S'$, and such that, for $t\leq t'$, $S_t$ is contained in 
the interior of $S_{t'}$ and that $S_t$ has an induced metric of non-positive curvature. 

For this we will use the fact that given a surface $S\subset M$ is associated to 
a metric $h$ in the conformal class at infinity, then the curvature of $h$ has the
same sign as the curvature of $S$ at the corresponding point (the correspondence
being through the hyperbolic Gauss map). This is because the curvature of $h$
is equal to $K_h=\frac K{\det(E+B)}$,
where $K$ is the curvature of the induced metric on $S$ (see \cite{volume}[Lemma 5.2]) and
$\det(E+B)\geq 0$ if $S$ corresponds to a metric $h$ at infinity.

Now consider the metrics at infinity $h, h'$ corresponding to $S,S'$ respectively.
They are conformal and $h\leq h'$ at each point by Lemma \ref{lm:hh'}, so we can write
$h'=e^{2u}h$ for a function $u:\partial M\to \R_{\geq 0}$. Then $K_{h'}=e^{-2u}(K_{h}+\Delta u)$
and both $K_h$ and $K_{h'}$ are non-positive, so $K_{h}\leq 0$ and $K_{h}+\Delta u\leq 0$.
For all $t\in [0,1]$ consider the metric $h_t=e^{2tu}h$. It is conformal to $h$ and $h'$,
with curvature $K_{h_t}=e^{-2tu}(K_{h}+t\Delta u)\leq 0$. So $h_t$ corresponds to a surface
$S_t$ with non-positive curvature. 
The monotonicity of $(h_t)$ and Lemma \ref{lm:hh'} show that the $S_t$ provide
a foliation of the domain of $M$ between $S$ and $S'$, as required.

It is now sufficient to prove that 
$$ \frac d{dt} W(S,S_t)\geq 0~, $$
with equality only if $S_t$ is stationary.

Consider now a fixed value of $t$, and suppose that the normal 
first order deformation of $S_t$ is given by $uN$, where $N$
is the unit exterior normal to $S_t$ and $u$ is a non-negative
function on $S_t$. We know (see \cite[Eq. (41)]{volume}) that
the first-order variation of $W(S_t)$ is given by
$$ \delta W(S_t) = \frac{1}{4}\int_{S_t} \delta H + \langle \delta I,
\II - \frac{H}{2} I\rangle da~, $$
where $\delta H$ is the first-order variation of $H$ and
$\delta I$ is the first-order variation of the induced metric.

Now a direct and classical computation shows that 
$$ \delta \II = -Hess(u) +u(\III+I)~, $$
while
$$ \delta I = 2u\II~. $$
It follows that 
$$ \delta H = \tr_I(\delta \II) - \langle \delta I, \II\rangle 
= \Delta u +2u -u\tr_I(\III)~. $$
Therefore
$$ \delta W(S_t) = \frac{1}{4} \int_{S_t} \Delta u + 2u +\langle u\II,
\II- H I\rangle da~, $$
so
$$ \delta W(S_t) = \frac{1}{4} \int_{S_t} \Delta u + 2u -2 u\det(B) da =
\frac{1}{4} \int_{S_t} \Delta u - 2uK da~. $$
But the integral of $\Delta u$ is zero while the other term is non-negative, 
and the result follows.
\end{proof}

The proof of Proposition \ref{pr:conf-change} clearly follows from this lemma.

\begin{remark}
Sergiu Moroianu pointed out that a simpler proof of Proposition \ref{pr:conf-change}
can be obtained once one knows that $W$ satisfies a ``Polyakov formula'' as in 
\cite{renormvol}[(1), p.2], a fact that we do not use or prove here but which
is true.
\end{remark}

\section{Proof of Theorem \ref{tm:compare}}

\subsection{The upper bound on $V_R$}

Let $\hgr$ be the ``grafting metric'' on $\dr_\infty M$ (it is also sometimes
called the Thurston metric). Recall
that $\hgr$ is a metric with curvature in $[-1,0]$ in the 
conformal class at infinity. In the simplest case where the support of $l$ is
a simple closed curve $c$, with a weight $w$, $\hgr$ is obtained by cutting 
$(\partial M, m)$ along the geodesic realizing $c$ and gluing in a flat
strip of width $w$.

\begin{lemma}
Let $m$ and $l$ be the induced metric and the measured bending lamination
on the boundary of the convex core of $M$. Then
$$ W(M,\hgr/2) = V_C(M) - \frac{1}{4}L_m(l)~. $$
\end{lemma}

\begin{proof}
We prove first that the metric at infinity corresponding to the foliation by
the equidistant surfaces from the convex core is $h_{gr}/2$. It is sufficient 
to do the proof when the measured bending lamination is along closed curves, 
since the general case then follows by density.

Let $S_\rho$ be the equidistant surface at distance $\rho$ from the convex core.
A standard computation in hyperbolic geometry shows that:
\begin{itemize}
\item on the parts of $S_\rho$ that project to the complement of the support of
the bending lamination $l$, the induced metric is $\cosh(\rho)^2m$, where $m$ 
is the pull-back  of the induced metric on $\partial C(M)$ on $S_\rho$ by the projection,
\item on the parts of $S_\rho$ projecting to the support of $l$, 
$I_\rho=\cosh(\rho)^2 m + \sinh(\rho)^2d\theta^2$, where $\theta$ is the angle variable 
on the normals to the support planes of $C(M)$ along $l$.
\end{itemize}

So it follows from Definition \ref{df:I*} that:
\begin{itemize}
\item on the parts of $S_\rho$ that project to the complement of the support of
$l$, 
$$ I^* = \lim_{\rho\to\infty} 2e^{-2\rho} \cosh(\rho)^2 m = \frac 12 m~, $$ 
\item on the parts of $S_\rho$ projecting to the support of $l$, 
$$ I^* = \lim_{\rho\to\infty} 2e^{-2\rho} (\cosh(\rho)^2 m + \sinh(\rho)^2d\theta^2)
= \frac 12 (m+d\theta^2)~. $$
\end{itemize}
This is precisely the metric $h_{gr}/2$.

The result therefore follows from the definition of $W(M,h)$, Definition \ref{df:WW}, 
for $r=0$.
\end{proof}

Let $h_0$ be the hyperbolic metric in the conformal class at infinity
of $M$. It follows from Lemma \ref{lm:conf-compare} that $\hgr\geq h_0$ at all points
of $\dr_\infty M$, and Proposition \ref{pr:conf-change} therefore indicates that 
$W(M,h_0)\leq W(M,\hgr)$. Since $V_R(M)=W(M,h_0)$, we find using Corollary \ref{cr:add} that
$$ V_R(M)\leq W(M,h_{gr}) = W(M,h_{gr}/2) -\pi\frac{\log 2}2 \chi(\partial M) = V_C(M) - L_m(l)/4 
+ \pi\frac{\log 2}2 |\chi(\partial M)|~. $$ 

\subsection{The lower bound on $V_R$}

The area of $\hgr$ is equal to $-2\pi \chi(\dr M) + L_m(l)$.
Therefore, the metric 
$$ \hgr' := \frac{-2\pi \chi(\dr M)}{-2\pi \chi(\dr M) + L_m(l)} \hgr $$
has area equal to $-2\pi \chi(\dr M)$, which is equal to the area of $h_0$. 
Since $V_R(M)$ is the maximum over $W(M,h)$ for $h$ a metric in the conformal
class at infinity of area equal to the area of $h_0$ (see \cite{volume,review}) we
find that $$ W(M, \hgr') \leq V_R(M)~. $$

However Corollary \ref{cr:add} indicates that 
$$ W(M, \hgr') = W(M,\hgr) - \pi 
\log\left(\frac{-2\pi \chi(\dr M)}{-2\pi \chi(\dr M) + L_m(l)}\right) \chi(\dr M)~. $$
It is also known that, if $M$ has incompressible boundary, then 
$L_m(l)\leq C(M)$. It follows that 
$$ W(M, \hgr) \leq V_R(M) + C'(M)~, $$
where $C'(M)$ is a constant which can easily be explicitly computed in terms 
of $C(M)$ and of $\chi(\partial M)$. This concludes the proof of Theorem \ref{tm:compare}.

\section{Proof of Theorem \ref{tm:upper}}

\subsection{The Bers embedding}
\label{ssc:bers}
We recall here a the basic setup of the Bers embedding. We consider 
a quasifuchsian hyperbolic 3-manifold $M\simeq S\times \R$, denote by $c_+$ and 
$c_-$ the complex structures at $+\infty$ and $-\infty$, respectively,
and by $\sigma_+$ and $\sigma_-$ the complex projective structures at infinity.
We also call $\sigma_+^F$, $\sigma_-^F$ the Fuchsian complex projective
structures with underlying complex structures $c_+, c_-$ on $S$. We
can then define two holomorphic quadratic differentials 
$$ q_-=\sigma_--\sigma_-^F~,~~ q_+=\sigma_+-\sigma_+^F~, $$
where the minus sign refers to the comparison of two complex 
projective structures on a given Riemann surface using the Schwarzian
derivative, see \cite{dumas-survey}.

Then, if $q$ is the holomorphic quadratic differential on $(\partial M, c)$ corresponding to
$q_\pm$ on the corresponding boundary component of $M$, we have (see \cite[Lemma 8.3]{volume})
$$ \II^*_0 = - Re(q)~, $$
that is, the real part of $q$ is minus the traceless part of the second fundamental form at infinity.

We now fix the conformal structure $c_-$ on the lower boundary at infinity
of $M$, and vary $c_+$. Each choice of $c_+$ determines a complex projective
structure $\sigma_-$ on the lower boundary at infinity of $M$, and
therefore a holomorphic quadratic differential $q_-\in Q_{c_-}$. This
defines a map $B_+:\cT_+\rightarrow Q_{c_-}$, called the Bers embedding.

Using the hyperbolic metric $h_-$ in the conformal
class of $c_-$, we can measure at each point of $S$ the norm of $q_-$.
We call $Q_{c_-}^\infty$ the vector space $Q_{c_-}$, endowed with this
$L^\infty$ norm. 

\begin{theorem}[Nehari] \label{tm:nehari}
The image of $B_+$ contains the ball of radius $1/2$, and is contained in the ball of 
radius $3/2$ in $Q^\infty_{c_-}$.
\end{theorem}

See \cite[Theorem 2.1]{McMullen} or \cite[Theorem 1, p. 134]{gardiner:quasiconformal}
(but note that in this reference the bound is given for a metric of constant curvature 
$-4$ on the disk).

Consider now on $Q_{c_-}$ the $L^2$-norm for the Weil-Petersson metric,
and denote by $Q_{c_-}^2$ the vector space $Q_{c_-}$ endowed with this
norm.

\begin{cor} \label{cr:nehari}
The image of $B_+$ is contained in the ball of radius $3\sqrt{\pi(g-1)}$
in $Q_{c_-}^2$.
\end{cor}

\begin{cor} \label{cr:II*}
For all $(c_-, c_+)\in \cT\times \bar \cT$, if $h_-$ is the hyperbolic metric in the 
conformal class of $c_-$, we have 
$$ \left(\int_S \| Re(q_-)\|_{h_-}^2 da_{h_-}\right)^{1/2}\leq 3\sqrt{2\pi (g-1)}~. $$ 
\end{cor}

\begin{proof}
Let $z$ be a complex coordinate at a point of $S$, let $h_-=\rho^2 |dz|^2$, and 
let $q_-=fdz^2$. The Nehari estimate above indicates that $|f|/\rho^2\leq 3/2$. 
However if $f=g+ih$ then $Re(q)=Re(fdz^2)=g(dx^2-dy^2)-2hdxdy$, so that 
$\| Re(fdz^2)\|_{h_-} = \sqrt{2} |f|/\rho^2$. Therefore $\| Re(fdz^2)\|_{h_-} 
\leq 3\sqrt 2/2$ pointwise, and the result follows.  
\end{proof}

\subsection{The upper bound}

We now prove Theorem \ref{tm:upper}. Let $c_-,c_+\in \cT$, and let $d=d_{WP}(c_-,c_+)$.

With the notations introduced here, Proposition \ref{pr:variation} can be written
as follows.

\begin{prop} \label{pr:variation2}
Under a first-order deformation of the hyperbolic structure on $M$,
$$ dV_R = \frac{1}{4} \left(\int_{\partial_-M}\langle Re(q_-), h'_-\rangle_{h_-} da_{h_-}
+ \int_{\partial_+M}\langle Re(q_+), h'_+\rangle_{h_+} da_{h_+}\right)~. $$
\end{prop}

\begin{proof}[Proof of Theorem \ref{tm:upper}]
Let $c:[0,d]\rightarrow \cT$ be the geodesic segment parameterized at constant 
velocity $1$ between $c_-$ and $c_+$. 
Integrate the equation in the previous proposition with $c_+$ replaced by $c(t)$, $t\in [0,d]$,
in Proposition \ref{pr:variation2} and with $h(t)$ the hyperbolic metric in the conformal class 
of $c(t)$. This shows that
\begin{equation}
  \label{eq:VRdiff}
 V_R(c_-,c_+)=\int_{t=0}^d \frac{1}{4} \int_{\partial_+M}\langle Re (q(t)),h'(t)\rangle_{h(t)} da_{h(t)}dt~,
\end{equation}
where $q(t)$ is the holomorphic quadratic differential equal to the Schwarzian differential
of the identity between the Fuchsian complex projective structure obtained from Riemann
uniformization from $c(t)$, and the quasifuchsian complex projective structure obtained
by applying the Bers double uniformization theorem to $(c_-,c(t))$. 

Denote by $g_{FT}$ the Fischer-Tromba metric on $\cT$, defined by 
$$ g_{FT}(H,H') = \int_S \langle H,H'\rangle_{h} da_h~, $$
where $H,H'$ are two symmetric 2-tensors on $S$. Recall that if $H,H'$ are two
traceless and divergence-free deformations of $h$ corresponding to tangent vectors
$C,C'\in T_c\cT$, then
$$ g_{FT}(H,H') = 8g_{WP}(C,C')~, $$
see e.g. \cite[Eq (3), p10]{cyclic2} for a computation of this relation.

Equation \eqref{eq:VRdiff} can be written
as
$$  V_R(c_-,c_+)=\int_{t=0}^d \frac{1}{4} g_{FT}(Re (q(t)),h'(t))dt~. $$
Therefore, using Corollary \ref{cr:II*},
$$  V_R(c_-,c_+)\leq \int_0^d \frac 14 \| Re(q)\|_{FT}\| h'(t)\|_{FT}dt
\leq \frac 34 \sqrt{2\pi(g-1)}\int_0^d\| h'(t)\|_{FT}dt~, $$
and it follows from the relation $8g_{WP}=g_{FT}$ that 
$$ V_R(c_-,c_+)\leq 3\sqrt{\pi(g-1)} d_{WP}(c_-,c_+)~. $$
\end{proof}

We provide here an alternate proof of the same statement, based on a different
(but obviously equivalent) computation.

\begin{proof}[Alternate proof of Theorem \ref{tm:upper}]
Let $c\in \cT_S$ be a complex structure on $S$, let 
$q$ and $\mu$ be a holomorphic quadratic differential and a Beltrami differential 
on $(S,c)$, and let $h'$ be the first-order variation corresponding to $\mu$
of the hyperbolic metric $h$ in the conformal class defined by $c$. Then a 
direct computation shows that
$$ \int_S \langle Re(q),h'\rangle_h da_h = 4Re\left(\int_S q\mu\rangle\right)~. $$
Applying this relation with $q$ equal to Schwarzian derivative term as above,
and using that $\II^*_0=-Re(q)$, we obtain that for a variation $h'$ of the
hyperbolic metric $h$ in the conformal class on the upper component of the boundary
at infinity, 
$$ dV_R(h') = -\frac 14\int_S \langle \II^*_0,h'\rangle_h da_h = 
\frac 14\int_S\langle Re(q),h'\rangle_h da_h =
Re\left( \int_S q\mu \right)~. $$

Let $z$ be a local complex coordinate, with $h=\rho^2|dz|^2$, then 
we can write
$$ q = q' dz^2~,~~ \mu = \mu' \frac{d\bar z}{dz}~, $$
so that 
$$ dV_R(h') = Re\left(\int_S \left(\frac{q'}{\rho^2}\right) \mu'
\rho^2 |dz|^2\right)~. $$
Using the Nehari estimate (Theorem \ref{tm:nehari}) shows that $|q'/\rho^2|\leq 3/2$, and it 
follows that
$$ |dV_R(h')|\leq \frac 32 \int_S \left| \mu' \right| \rho^2 |dz|^2~, $$
and it follows from the Cauchy-Schwarz inequality that
$$ |dV_R(h')|\leq \frac 32 \| \mu\|_{WP} \sqrt{4\pi (g-1)} = 3\sqrt{\pi (g-1)} \| \mu\|_{WP}~. $$

It is then possible to integrate this inequality on a path from $c_-$ to $c_+$ as in
the first proof above to obtain the result.
\end{proof}

\section{The size of almost flat holomorphic disks}
\label{sc:disk}

In this section we prove Theorem \ref{tm:disk}, giving an upper bound on the radius
of holomorphic disk in $\cT_S$ which are flat enough. The proof is based on a well-known
upper bound on the curvature of $g_{WP}$ and on two key
properties of the renormalized volume, as collected in the next lemma.

\begin{lemma} \label{lm:laplacien}
Let $D$ be a holomorphic disk immersed in $\cT_S$, with induced metric $q$, 
with center $c_0$. Consider the function $u:D\rightarrow \R$ defined by 
$u(c)=V_R(c_0,c)$ for all $c\in D$. Then
\begin{enumerate}
\item $q$ has negative curvature,
\item $\| du\|_q\leq 3\sqrt{\pi(g-1)}$,
\item $\Delta_q u=-2$.
\end{enumerate}
\end{lemma}

\begin{proof}
The first point follows from the fact that the Weil-Petersson metric on $\cT_S$ has
negative sectional curvature \cite{wolpert:chern} and from the Gauss formula, which
indicates that the curvature of a holomorphic disk in a K\"ahler manifold is at most equal
to the sectional curvature of the ambiant metric on its tangent space.

The second point follows from Proposition \ref{pr:variation} and from Corollary
\ref{cr:nehari}.

For the third point recall that the renormalized volume $V_R(c_-,\cdot)$, considered
as a function on $\cT_S$, is a K\"ahler potential for the Weil-Petersson metric on $cT_S$, see e.g.
\cite{TZ-schottky,takhtajan-teo,volume,review}.
In other terms:
$$ 2\partial \overline{\partial} V_R(c_-,\cdot) = i\omega_{WP}~. $$
But the restriction of $\partial \overline{\partial}$ to $D$ is $\partial \overline{\partial}$,
so equal to $-(1/4)\Delta_q$, where $\Delta_q$ is the Laplace operator of $(D,q)$.
\end{proof}

\begin{lemma} \label{lm:disk}
There is continuous, increasing function $\phi:[0,1)\rightarrow \R_{\geq 0}$ with $\phi(0)=0$ and
$\lim_1\phi=\infty$ as follows.
Let $(D,q)$ be a Riemannian disk of center $c_0$, and let $u:D\rightarrow \R$ be a smooth function. 
Suppose that:
\begin{itemize}
\item the radius of $(D,q)$ is $\phi(\delta)$,
\item $\Delta_q u=1$,
\item the curvature $K_q$ of $q$ is in $[-1,0]$.
\end{itemize}
Then there is a point $x\in D$ where $\| du\|_q\geq \delta$.
\end{lemma}

The precise value of the function $\phi$ can be obtained by solving a differential equation. 

\begin{proof}
For $r\in (0,R]$ the geodesic disk $B(r)$ of center $c_0$ and radius $r$ is convex. We denote by $A(r)$ 
its area, by $L(r)$ the length of its boundary, by $\kappa(r)$ the total curvature of its boundary, 
and by $\Kb(r)$ the mean of its curvature. By definition, $\Kb(r)\in [-1,0]$ for all $r\in (0,R]$. 
Moreover:
\begin{itemize}
\item $A'(r)=L(r)$,
\item $L'(r)=\kappa(r)$,
\item $A(r)\Kb(r)=2\pi-\kappa(r)$ by the Gauss-Bonnet theorem.
\end{itemize}
It follows that 
$$ L'(r)= 2\pi-\Kb(r)A(r)~. $$

Let $y(r)=A(r)/L(r)$. Then 
\begin{eqnarray*}
y'(r) & = & \frac{A'(r)L(r)-A(r)L'(r)}{L(r)^2} = \frac{L(r)^2-A(r)(2\pi-A(r)\Kb(r))}{L(r)^2} \\
& = & 1+\left(\Kb(r)-\frac{2\pi}{A(r)}\right)y(r)^2~.
\end{eqnarray*}
The initial condition is $\lim_0 y=0$ since $A(r)\sim \pi r^2$ and $L(r)\sim 2\pi r$ at $0$.
Since $q$ has curvature in $[-1,0]$, $A(r)\geq \pi r^2$ and $\Kb(r)\geq -1$ for all $r$, so that
\begin{equation}
  \label{eq:riccati}
   y'(r) \geq 1 - \left(1+\frac{2}{r^2}\right)y(r)^2~. 
\end{equation}
So $y(r)\geq y_0(r)$, where $y_0$ is the solution vanishing at $0$ of the
equation obtaining by taking the equality in (\ref{eq:riccati}). 

Let $\ub(r)$ be the mean of $u$ over $\partial B(r)$. Then 
$$ \ub'(r) = \frac{1}{L(r)}\int_{\partial B(r)} du(n) = \frac{1}{L(r)}\int_{B(r)} \Delta_q uda =
\frac{A(r)}{L(r)} = y(r)~. $$
It follows that there exists a point at distance $r$ from $c$ where $\partial u/\partial r\geq y(r)$,
and therefore where $\| du\|_q\geq y(r)\geq y_0(r)$. 

The lemma follows, with $\phi$ equal to the reciprocal of $y_0$.
\end{proof}

\begin{cor} \label{cr:disk}
Let $\Delta_0,k,\delta>0$.
Let $(D,q)$ be a Riemannian disk of center $c_0$, and let $u:D\rightarrow \R$ be a smooth function. 
Suppose that:
\begin{itemize}
\item the radius of $(D,q)$ is $\phi(k^2\delta/\Delta_0)/k$,
\item $\Delta_q u=\Delta_0$,
\item the curvature $K_q$ of $q$ is in $[-k^2,0]$.
\end{itemize}
Then there is a point $x\in D$ where $\| du\|_q\geq \delta$.
\end{cor}

\begin{proof}
The statement is obtained by scaling the metric $q$ by a factor $k^2$ and the function $u$ by a factor $k^2/\Delta_0$ 
in Lemma \ref{lm:disk}.
\end{proof}

\begin{proof}[Proof of Theorem \ref{tm:disk}]
It follows directly from Lemma \ref{lm:laplacien} and from Corollary \ref{cr:disk}.
The function $\phi$ is inverse function of the solution of the differential equation
$$ y'(r) = 1 - \left(1+\frac{2}{r^2}\right)y(r)^2 $$
which vanishes at $0$.  
An easy asymptotic analysis shows that $\lim_\infty y=1$, while $y'(0)=1/2$. It follows
that $\phi$ is defined on $[0,1)$ with $\lim_1 \phi=\infty$, and that $\phi'(0)=2$.
\end{proof}

\bibliographystyle{plain}

\def\cprime{$'$} \def\cprime{$'$}

\end{document}